\documentclass[11pt,reqno]{amsart}
\usepackage[T1]{fontenc}
\usepackage[utf8]{inputenc}
\usepackage{amsmath,amssymb,amsthm,mathtools}
\usepackage[margin=1.1in]{geometry}
\usepackage[expansion=false]{microtype}
\usepackage{enumitem}
\usepackage{booktabs,array}
\usepackage{listings}
\usepackage{xcolor}
\usepackage[colorlinks=true,linkcolor=blue!55!black,citecolor=green!45!black,urlcolor=blue!60!black]{hyperref}
\usepackage[capitalise,nameinlink]{cleveref}

\theoremstyle{plain}
\newtheorem{theorem}{Theorem}[section]
\newtheorem{proposition}[theorem]{Proposition}
\newtheorem{lemma}[theorem]{Lemma}

\newtheorem{problem}[theorem]{Problem}
\theoremstyle{definition}
\newtheorem{definition}[theorem]{Definition}

\newtheorem{remark}[theorem]{Remark}

\DeclareMathOperator{\tr}{tr}
\DeclareMathOperator{\Res}{Res}
\DeclareMathOperator{\Disc}{Disc}
\DeclareMathOperator{\Cov}{Cov}
\DeclareMathOperator{\Inv}{Inv}

\DeclareMathOperator{\res}{res}
\newcommand{\PP}{\mathbb{P}}
\renewcommand{\AA}{\mathbb{A}}
\newcommand{\QQ}{\mathbb{Q}}
\newcommand{\ZZ}{\mathbb{Z}}
\newcommand{\Fbar}{\overline{F}}
\newcommand{\GL}{\mathrm{GL}}
\newcommand{\SL}{\mathrm{SL}}
\newcommand{\PGL}{\mathrm{PGL}}
\newcommand{\Dop}{\mathcal D}
\newcommand{\Lop}{\mathcal L}
\newcommand{\Lam}{\Lambda}
\newcommand{\charr}{\operatorname{char}}
\newcommand{\cont}{\mathfrak c}
\usepackage{verbatim}

\begin{document}

\title[The Hermite--Joubert problem in degree seven]
{A septic covariant and the Hermite--Joubert problem in degree seven}

\author{Sunil K. Chebolu}
\address[Chebolu]{Department of Mathematics, Illinois State University, Normal, IL 61761, USA}
\email{schebol@ilstu.edu}

\author{J\'an Min\'a\v{c}}
\address[Min\'a\v{c}]{Department of Mathematics, The University of Western Ontario, London, ON N6A 5B7, Canada}
\email{minac@uwo.ca}

\author{Behzad Nikzad}
\address[Nikzad]{Tempered AI, 1920 Yonge St \#200, Toronto, ON M4S 3E2, Canada}
\email{behzad.nikzad@tempered.ai}

\author{Charlotte Ure}
\address[Ure]{Department of Mathematics, Illinois State University, Normal, IL 61761, USA}
\email{cure@ilstu.edu}


\begin{abstract}
We prove the degree-seven case of the Hermite--Joubert problem in characteristic zero:
if $F$ is a field of characteristic zero and $E/F$ is a field extension of degree seven,
then $E$ is generated by an element $a$ with $\tr_{E/F}(a)=\tr_{E/F}(a^{3})=0$, that is, with
minimal polynomial of the form $\lambda^{7}+c_{2}\lambda^{5}+c_{4}\lambda^{3}+c_{5}\lambda^{2}+c_{6}\lambda+c_{7}$, where the $c_i$'s belong to $F$.  This is given by an explicit formula: a covariant of the binary septic of
coefficient degree seven and order five, 
evaluated at a generator $\theta$ and divided by
the derivative of its minimal polynomial evaluated at $\theta$.
We also announce the general theorem, which will be proved in a
companion paper in preparation: over every infinite field, in every characteristic,
every \'etale algebra of degree seven contains a primitive element $a$ with
$c_{1}(a)=c_{3}(a)=0$. Moreover, every field extension of degree seven of an arbitrary field has a
generator $a$ with $c_{1}(a)=c_{3}(a)=0$.
\end{abstract}

\maketitle

\section{Introduction}\label{sec:intro}

Let $F$ be a field and $E/F$ an \'etale algebra of degree $n$ --- for instance a separable
field extension of degree $n$. For $a\in E$ write
\[
  \chi_{a}(\lambda)=\det(\lambda-a\mid_E)=\lambda^{n}+c_{1}(a)\lambda^{n-1}+\cdots+c_{n}(a)
\]
for the characteristic polynomial, where $c_{1}(a)=-\tr_{E/F}(a)$. We call $a$
\emph{primitive} if $E$ is generated by $a$ over $F$, that is, $F[a]=E$. When $n$ is invertible in $F$, the coefficient $c_{1}$ can be killed by the scalar translation
$a\mapsto a-\tfrac{1}{n}\tr_{E/F}(a)$. 
The Hermite--Joubert problem asks whether \emph{two} coefficients can be eliminated at once. 

\begin{problem}[Hermite--Joubert]\label[problem]{prob:HJ}
For which integers $n$, and over which classes of base fields $F$, does every \'etale
$F$-algebra $E$ of degree $n$ contain a primitive element $a$ with $c_{1}(a)=c_{3}(a)=0$,
that is, with characteristic polynomial of the form
$\lambda^{n}+c_{2}\lambda^{n-2}+c_{4}\lambda^{n-4}+c_{5}\lambda^{n-5}+\cdots+c_{n}$?
\end{problem}

When $\charr F\neq3$, Newton's identities turn the coefficient condition into
$\tr_{E/F}(a)=\tr_{E/F}(a^{3})=0$. The problem dates back to 1861 and 1867, when, in the classical characteristic-zero setting, Hermite
\cite{Hermite1861} proved the degree-five theorem and Joubert \cite{Joubert1867} the
degree-six theorem. Modern extensions require the familiar characteristic restrictions,
including $\charr F\ne2$ in degree six \cite{Reichstein2014}.
Further proofs were given by Coray \cite{Coray1987}  through the arithmetic of cubic
hypersurfaces, by Kraft \cite{Kraft2006} through classical invariant theory, and,
for fields containing an algebraically closed field, by Brassil and
Reichstein \cite{BrassilReichstein2019}   through Galois cohomology. Brassil and Reichstein \cite{BrassilReichstein2017} 
also provide positive results, for special fields and conditionally
over all fields, for Hermite--Joubert problems which include $n = 7$.
In the negative
direction, for the general degree-$n$ extension in characteristic zero, Reichstein
\cite{Reichstein1999} proved failure for $n=3^{m}$ and $n=3^{m}+3^{l}$ ($m>l\ge0$),
and Nguyen \cite{Nguyen2019}, addressing a conjecture of Brassil and Reichstein
\cite{BrassilReichstein2017}, extended this to $n=3^{k_{1}}+3^{k_{2}}+3^{k_{3}}$ with distinct $k_i$. The obstruction theorems
just cited concern sums of distinct powers of $3$, whereas $7=3+3+1$ is not of that form.
Coray \cite[Thm.~4.2]{Coray1987}  settled $n=7$ (and $n=8$) in the
affirmative for fields complete with respect to a discrete
valuation whose residue field has the Cassels--Swinnerton-Dyer
property and whose characteristic is different from 2 and 3, a class containing all local fields. Over a general
base field the degree-seven case remained open; Brassil and
Reichstein wrote in 2017: \emph{``for other values of $n$ (in
particular, for $n=7$), this question remains open''} \cite{BrassilReichstein2017}.

The purpose of this note is to announce a positive answer for $n=7$, and to give the
complete proof in the case of characteristic zero. The answer is constructive, with an explicit formula. For a binary septic
$f=\sum_{i=0}^{7}a_{i}X^{7-i}Z^{i} \in F[X,Z]$ define
\[ 
Q_{f}=7\,D_{2}P_{5}-10\,C_{1}T_{1},
\] with the unnormalised transvectant of
\Cref{def:tv}
\begin{equation}\label{eq:chain}
\begin{aligned}
  D_{1}&=(f,f)_{4}, & D_{2}&=(f,f)_{6}, & D_{3}&=(f,f)_{2},\\
  T_{1}&=(f,D_{1})_{4}, & T_{5}&=(f,D_{3})_{5}, & &\\
  C_{1}&=(f,T_{5})_{7}, & C_{7}&=(f,T_{1})_{3}, & P_{5}&=(f,C_{7})_{5}.
\end{aligned}
\end{equation}
By construction, $Q_{f}$ is a covariant of the binary septic of coefficient degree seven (in the 8 coefficients $a_{0},\dots,a_{7}$ of $f$) and order five (in the two variables $X$ and $Z$). Its full expansion in $\ZZ[a_{0},\dots,a_{7},X,Z]$ has $780$ nonzero monomials; its coefficients have greatest common divisor
$\cont=2^{32}3^{14}5^{6}\cdot7$, and $\widehat Q_{f}=Q_{f}/\cont$ is the primitive integral
covariant. Put $q_{f}(T)=\widehat Q_{f}(T,1)$, a polynomial of degree at most $5$.

The recoupling identities of  Lemma \ref{prop:recoup}  reduce $Q_f$ from its original
$780$ terms to an expression in $f$ and its quadratic covariant alone;
and after a projective change of variable it collapses further, to
four terms (Proposition \ref{prop:collapse}).

The following is then the main theorem of the paper that resolves the Hermite--Joubert problem affirmatively in degree seven over fields of characteristic zero. Throughout our paper, except when we refer to results in a forthcoming paper,
we assume that our base field $F$ has characteristic zero.

\begin{theorem}\label[theorem]{thm:main}
Let $F$ be a field of characteristic zero and $E/F$ a field extension of degree seven.
Then there is a generator $a$ of $E/F$ with $\tr_{E/F}(a)=\tr_{E/F}(a^{3})=0$. Equivalently,
the minimal polynomial of $a$ has the form
\[
  \lambda^{7}+c_{2}\lambda^{5}+c_{4}\lambda^{3}+c_{5}\lambda^{2}+c_{6}\lambda+c_{7}.
\]
More precisely, for every generator $\theta$ of $E/F$ with minimal polynomial $f_{\theta}$
the element
\begin{equation}\label{eq:formula}
  a=\frac{q_{f_{\theta}}(\theta)}{f_{\theta}'(\theta)}
\end{equation}
satisfies $\tr_{E/F}(a)=\tr_{E/F}(a^{3})=0$, and there is a nonempty Zariski-open subset $\mathcal{W}$ of the
affine space underlying $E$, consisting of generators, such that $a\ne0$ for every
$\theta\in \mathcal{W}(F)$. Since $F$ is infinite,  $\mathcal{W}(F)\ne\emptyset$, and every such $a$ generates $E/F$.
\end{theorem}

The first part of \Cref{thm:main}  follows from two identities: a classical linear one provided by Euler, and a highly non-trivial cubic one that forms the core mathematical engine of this paper. We record these in the next theorem.

\begin{theorem}\label[theorem]{thm:identity}
Let $f$ be a monic separable polynomial of degree seven over a field of characteristic zero,
with roots $r_{1},\dots,r_{7}$ in an algebraic closure, and put
$y_{i}=q_{f}(r_{i})/f'(r_{i})$. Then
\[
  \sum_{i=1}^{7}y_{i}=0\qquad\text{and}\qquad\sum_{i=1}^{7}y_{i}^{3}=0 .
\]
\end{theorem}

The first sum is Euler's identity. For the proof of the second identity (\Cref{sec:proof}) we remark the following: after a projective change of
variable that moves the two roots of the quadratic covariant $D_{2}$ to $0$ and $\infty$, 
the covariant $Q_f$, whose generic expansion has $780$ monomials,  collapses to a constant multiple of the four-term
polynomial $q=a_{1}T^{5}-a_{3}T^{3}-a_{4}T^{2}+a_{6}$, and for this $q$ one writes down an
explicit polynomial $g$ of degree at most $11$ with $W(f',g)=f'g'-f''g\equiv q^{3}\pmod f$. Since
$W(f',g)(r)/f'(r)^{3}$ is the residue at $r$ of the differential $g\,dT/f^{2}$, and this
differential has no residue at infinity when $\deg g\le12$, the cubic sum vanishes by the
residue theorem.

\subsection*{Forthcoming work}
This note proves the characteristic-zero field-extension case. A comprehensive companion
paper in preparation will prove the general theorem: (i) after primitive integral
normalisation the septic identity becomes an identity over $\ZZ$, hence holds in every
characteristic; (ii) over every infinite field $F$, in every characteristic, every
\'etale $F$-algebra of degree seven contains a primitive element $a$ with
$c_{1}(a)=c_{3}(a)=0$, and every field extension of degree seven of an arbitrary field has
a generator that satisfies $c_{1}(a)=c_{3}(a)=0$. (Over a finite field a split algebra of degree seven
need not have a primitive element at all, so ``infinite'' cannot be dropped in the \'etale
statement.) The companion paper also studies the minimality of the coefficient degree
seven within the order-five covariant  and the rigidity of $Q_{f}$. It will explain the ratio $7:-10$ as the spectral value $10$ of an Euler operator and the
Rankin--Cohen-type calculus behind the polynomial $g$.   The methods and ideas developed in this paper provide a common framework for studying Hermite--Joubert type problems and suggest a treatment of every degree greater than or equal to 5, except $9$, over $\QQ$.

\subsection*{Organization}
\Cref{sec:dictionary} provides a preliminary discussion on trace conditions and sums over
roots. In this section, we parametrize all elements in our degree 7 extension that have trace zero. \Cref{sec:covariant} defines transvectants and the key covariant $Q_{f}$.
\Cref{sec:proof} and \Cref{sec:mainproof}  prove  \Cref{thm:identity} and  
\Cref{thm:main} respectively. In \Cref{sec:example}, we provide concrete examples over the rationals to illustrate our methods.
\Cref{sec:code} describes the computational aspects of our work, and \Cref{sec:ai} is a
declaration on the use of artificial intelligence in this work.

\subsection*{Acknowledgements}
We thank Z. Reichstein and M. Brassil for formulating the problem in the shape that made this work
possible. We thank Nguyen Duy Tan for discussions concerning this project. J\'an Min\'a\v{c} is partially supported by the Natural Sciences
and Engineering Research Council of Canada (NSERC) grant R0370A01; he also gratefully
acknowledges the Western University Faculty of Science Distinguished Professorship  and the support of the Western Academy for Advanced Research. Charlotte Ure was partially supported by a Pre-Tenure Faculty Initiative Grant (PFIG) from Illinois State University.

\section{Trace conditions as sums over roots}\label{sec:dictionary}

Throughout this section $f\in F[T]$ is monic
separable (that is, its discriminant is nonzero) of degree $n$ with roots $r_{1},\dots,r_{n}\in\Fbar$, and $E=F[T]/(f)$ with
$\theta$ the class of $T$. We remark that $E$ is an \'etale algebra over $F$, not necessarily a field extension. In the following, we denote by $\tr(-)$ the trace $\tr_{E/F}(-)$ and we use $\chi_a(-)$ for the characteristic polynomial of multiplication by $a$ on $E$. For $h\in F[T]$ we define $a=h(\theta)$, so that 
\begin{equation}\label{eq:trsum}
  \tr(a^{k})=\sum_{i=1}^{n}h(r_{i})^{k},\qquad
  \chi_{a}(\lambda)=\prod_{i=1}^{n}\bigl(\lambda-h(r_{i})\bigr).
\end{equation}
For an in-depth discussion on this material see \cite[Ch. V, \S 8, no. 2]{Bourbaki1990}. 

We recall the power sums
\[
 p_k(a)=\sum_{i=1}^n \left( h(r_i)\right)^k=\tr(a^k)
\]
and we denote by $e_k$ the elementary symmetric functions so that
\[
 e_k(a)=\sum_{1 \leq i _1 < i_2 < \ldots < i_k \leq n } h(r_{i_1}) h(r_{i_2}) \cdots h(r_{i_k}). 
\]
Newton's identity $p_{3}=e_{1}^{3}-3e_{1}e_{2}+3e_{3}$ shows that $c_{1}(a)=c_{3}(a)=0$ if and only if $\tr(a)=\tr(a^{3})=0$.  For a detailed history on Newton's identity, see for instance \cite{Minac2003}. The following formula is classical. 

\begin{lemma}\label[lemma]{lem:euler}
For $0\le k\le n-2$, the sum $\displaystyle\sum_{i=1}^{n}\frac{r_{i}^{k}}{f'(r_{i})}=0$.
Consequently $\tr\bigl(u(\theta)/f'(\theta)\bigr)=0$ for every $u\in F[T]$ with
$\deg u\le n-2$.
\end{lemma}

\begin{proof}
Partial fractions give 
\[\dfrac{T^{k}}{f(T)}=\sum_{i=1}^n\dfrac{r_{i}^{k}}{f'(r_{i})}\cdot\dfrac{1}{T-r_{i}}.\]
We multiply by $T$ and let $T\to\infty$. The left side tends to $0$ when $k\le n-2$, and the right
side to $\sum_{i}r_{i}^{k}/f'(r_{i})$. The second statement follows from \eqref{eq:trsum} applied to $h\equiv u\tilde{h} \pmod f$, where $\tilde{h}\in F[T]$ satisfies $\tilde{h}f'\equiv1\pmod f$. We know such a $\tilde{h}$ exists because $\gcd(f,f')=1$.
\end{proof}

Thus every $u$ of degree $\le n-2$ produces an element $a=u(\theta)/f'(\theta)$ with
$\tr(a)=0$ (and conversely every trace-zero element arises this way, since
the map $u \mapsto u(\theta)/f'(\theta)$ is injective and both spaces have dimension $n-1$ over $F$, as the trace form of an étale algebra is nondegenerate). The whole problem is
to choose $u$, as a function of the coefficients of $f$, so that also $\sum_{i}y_{i}^{3}=0$
for $y_{i}=u(r_{i})/f'(r_{i})$. The appropriate choice of $u$ will be derived using classical invariant theory. 

\section{The septic covariant}\label{sec:covariant}

A \emph{binary form} of order $m$ over $F$ is a homogeneous polynomial
$\varphi(X,Z)=\sum_{i=0}^{m}\varphi_{i}X^{m-i}Z^{i}\in F[X,Z]$. We identify a monic $f(T)$ of degree
$n$ with the form $f(X,Z)=\prod_{i}(X-r_{i}Z)$, so that $f(T,1)=f(T)$. A \emph{covariant} of
the binary $n$-ic of coefficient degree $d$ and order $m$ is a polynomial $\Phi(a;X,Z)$,
homogeneous of degree $d$ in the coefficients $a=(a_{0},\dots,a_{n})$ and of degree $m$ in
$(X,Z)$, such that $\Phi(g\cdot a;X,Z)=\Phi(a;g^{-1}(X,Z))$ for all $g\in\SL_{2}$ over any extension of $F$, where
$(g\cdot f)(X,Z)=f(g^{-1}(X,Z))$. In words: a covariant attaches to a configuration of $n$
points of the projective line
 a configuration of $m$ points, compatibly with projective transformations. We write $\Cov_{d,m}(n)$ for the space of these, and
$\Inv_{d}(n)=\Cov_{d,0}(n)$. 

\begin{definition}[Transvectant]\label[definition]{def:tv}
For forms $A$ of order $\mu$ and $B$ of order $\nu$ and $0\le r\le\min(\mu,\nu)$, define the (unnormalised) transvectant
\[
  (A,B)_{r}=\sum_{j=0}^{r}(-1)^{j}\binom rj\,
  \frac{\partial^{r}A}{\partial X^{r-j}\partial Z^{j}}\cdot
  \frac{\partial^{r}B}{\partial X^{j}\partial Z^{r-j}} .
\]
\end{definition}

The transvectant of two covariants is a covariant of order $\mu+\nu-2r$ and coefficient
degree the sum of the two degrees; $(A,B)_{r}=(-1)^{r}(B,A)_{r}$, so $(f,f)_{r}=0$ for odd
$r$. Two
classical facts will be used \cite{Olver1999,Springer1977}: transvectants of covariants are covariants, and the dimensions are given by the Cayley--Sylvester formula; the weight $w$ below is the bookkeeping index $(nd-m)/2$.

\begin{equation}\label{eq:CS}
  \dim\Cov_{d,m}(n)=N(n,d,w)-N(n,d,w-1),\qquad w=\tfrac{nd-m}{2},
\end{equation}

where $N(n,d,w)$ is the number of multisets $\{i_{1}\le\dots\le i_{d}\}\subseteq\{0,\dots,n\}$
with $\sum i_{j}=w$ (the dimension is $0$ if $nd-m$ is odd).

With $n=7$ the members of the chain \eqref{eq:chain} have (degree, order)
\[
  D_{1}:(2,6),\ \ D_{2}:(2,2),\ \ D_{3}:(2,10),\ \ T_{1}:(3,5),\ \ T_{5}:(3,7),\ \
  C_{1}:(4,0),\ \ C_{7}:(4,6),\ \ P_{5}:(5,3),
\]
so that $Q_{f}=7D_{2}P_{5}-10C_{1}T_{1}$ has bidegree $(7,5)$. Using  \eqref{eq:CS}, we remark that
\begin{equation} \label{eq:dims}
\begin{aligned}
  \dim\Cov_{2,2}(7) &= 4-3 = 1, & \dim\Cov_{3,5}(7)&=9-8=1,\\
  \dim\Inv_{4}(7)&=24-23=1,& \dim\Cov_{5,3}(7)&=48-46=2 .
\end{aligned}
\end{equation}
The quadratic covariant $D_{2}$ is, up to a scalar, the unique one of degree two. More precisely
\begin{equation}\label{eq:D2}
  D_{2}=207360\,(A X^{2}+B XZ+C Z^{2}),\qquad
  \begin{aligned}
   A&=35a_{0}a_{6}-10a_{1}a_{5}+5a_{2}a_{4}-2a_{3}^{2},\\
   B&=245a_{0}a_{7}-25a_{1}a_{6}+5a_{2}a_{5}-a_{3}a_{4},\\
   C&=35a_{1}a_{7}-10a_{2}a_{6}+5a_{3}a_{5}-2a_{4}^{2}.
  \end{aligned}
\end{equation}

Write the roots as vectors
$v_{i}=(x_{i},z_{i})$, so $f=\prod_{i}(z_{i}X-x_{i}Z)$, put $[v,w]=x_{v}z_{w}-x_{w}z_{v}$
and, for $Q\in\Cov_{d,n-2}(n)$, $\widetilde y_{i}=Q(v_{i})/\prod_{j\ne i}[v_{i},v_{j}]$; when
all $z_{i}=1$ this is $y_{i}=Q(r_{i},1)/f'(r_{i})$.

\begin{lemma}\label[lemma]{lem:inv}
Let $Q\in\Cov_{d,n-2}(n)$.
\begin{enumerate}[label=\textup{(\alph*)},leftmargin=2.2em]
\item Rescaling $v_{i}\mapsto\lambda_{i}v_{i}$ multiplies every $\widetilde y_{j}$ by the same
  scalar $(\prod_{i}\lambda_{i})^{d-1}$; replacing every $v_{i}$ by $gv_{i}$ with $g\in\SL_{2}$
  leaves every $\widetilde y_{i}$ unchanged. Consequently, over $\Fbar$, a projective change
  of coordinates of the root configuration replaces $(y_{1},\dots,y_{n})$ by a permutation
  of itself times one common nonzero scalar, and the conditions $\sum y_{i}=0$,
  $\sum y_{i}^{3}=0$ are $\PGL_{2}(\Fbar)$-invariant.
\item For monic separable $f$, $e_{3}(y)\cdot\prod_{i}f'(r_{i})=\sum_{|S|=3}\prod_{i\in S}Q(r_{i},1)\prod_{i\notin S}f'(r_{i})$
  is a symmetric polynomial in the roots, hence a polynomial in $a_{1},\dots,a_{n}$; and
  $\prod_{i}f'(r_{i})=\pm\Disc(f)$.
\end{enumerate}
\end{lemma}

\begin{proof}
\begin{enumerate}[label=\textup{(\alph*)},leftmargin=2.2em]
\item Rescaling multiplies $f$ by $\Lambda=\prod\lambda_{i}$, hence $Q$ by $\Lambda^{d}$. Evaluating the order-$(n-2)$ form $Q$ at $\lambda_{j}v_{j}$ contributes $\lambda_{j}^{n-2}$,
while $\prod_{i\ne j}[\lambda_{j}v_{j},\lambda_{i}v_{i}]=\lambda_{j}^{n-1}\Lambda\lambda_{j}^{-1}\prod_{i\ne j}[v_{j},v_{i}]$
contributes $\lambda_{j}^{n-2}\Lambda$. The ratio is $\Lambda^{d-1}$. Brackets are
$\SL_{2}$-invariant and $Q$ is covariant, which gives the second statement. Over $\Fbar$
every $g\in\GL_{2}$ is a scalar times an element of $\SL_{2}$, and renormalising the
transformed form to be monic is a rescaling. 
\item In each monomial of $e_{3}(y)$ the three
indices are distinct, so multiplying by $\prod_{i}f'(r_{i})$ clears every denominator once. Hence we have a polynomial in $a_1, \cdots, a_n$, and the symmetry is clear.  \qedhere
\end{enumerate}
\end{proof}

\section{Proof of the septic identity}\label{sec:proof}

We work over an algebraically closed field $\Fbar$. For
$u,v\in\Fbar[T]$ let $W(u,v)=uv'-u'v$ be the Wronskian; note $W(u,v)=u^{2}(v/u)'$. We begin our investigation by computing residues of polynomials. 

\begin{lemma}\label[lemma]{lem:residues}
Let $f$ be separable of degree $n$ with roots $r_{i}$, and let $u,g\in\Fbar[T]$.
\begin{enumerate}[label=\textup{(\alph*)},leftmargin=2.2em]
\item $\res_{T=r_{i}}\dfrac{u\,dT}{f}=\dfrac{u(r_{i})}{f'(r_{i})}$ and
  $\res_{T=r_{i}}\dfrac{g\,dT}{f^{2}}=\dfrac{W(f',g)(r_{i})}{f'(r_{i})^{3}}$.
\item If $\deg g\le2n-2$, then $\displaystyle\sum_{i=1}^{n}\frac{W(f',g)(r_{i})}{f'(r_{i})^{3}}=0$.
\end{enumerate}
\end{lemma}

\begin{proof}
\begin{enumerate}[label=\textup{(\alph*)},leftmargin=2.2em]
\item Write $f=(T-r)v$ with $v(r)\ne0$; then $f'(r)=v(r)$ and $f''(r)=2v'(r)$. The residue of
$u/f$ at $r$ is $u(r)/v(r)$. The residue of $g/f^{2}=(T-r)^{-2}g/v^{2}$ at $r$ is
$(g/v^{2})'(r)=\bigl(g'(r)v(r)-2g(r)v'(r)\bigr)/v(r)^{3}=\bigl(f'g'-f''g\bigr)(r)/f'(r)^{3}$.
\item The residues of a rational differential on $\PP^{1}$ sum to zero. The finite poles of
$g\,dT/f^{2}$ are the roots of $f$, with residues given by (a); at infinity, $g/f^{2}=O(T^{-2})$
when $\deg g\le2n-2$, so with $s=1/T$ the differential is $-g(1/s)s^{-2}f(1/s)^{-2}ds$, a
power series in $s$ times $ds$, with residue $0$. \qedhere
\end{enumerate}
\end{proof}

As a result of the previous lemma, we need to determine $g$ with $\deg g\le12 = 2(7)-2$ and
$W(f',g)\equiv q^{3}\pmod f$ for some $q$. We will show in future work that such a $g$ exists if and only if the identity holds.  The difficulty in finding $g$ is that
$Q_{f}$, expanded in $a_0, \ldots, a_7$, has $780$ terms. We put $D_2$ in normal form, which will enable us to simplify $Q_f$. 

\begin{lemma} \label[lemma]{lem:normal}
On the dense open set of septics for which $D_{2}$ has two distinct roots, a projective
change of variables puts $D_{2}$ in the form $\lambda XZ$, $\lambda\ne0$; equivalently
$A=C=0$ in \eqref{eq:D2}, and then $\lambda=207360\,B$.
\end{lemma}

\begin{proof}
Move the two roots of the binary quadratic $D_{2}$ to $0$ and $\infty$; since $D_{2}$ is a
covariant, the transformed quadratic is a nonzero multiple of $XZ$.
\end{proof}

Let $\Dop=T\,\frac{d}{dT}$ be the Euler operator and $\Lop=\Dop(7-\Dop)$. We remark that for $0 \leq m \leq 7$, each monomial $T^m$ is an eigenvector of $\Lop$ as $\Lop T^{m}=m(7-m)T^{m}$. The corresponding eigenvalues are $0,6,10$, and $12$, each occurring with multiplicity $2$.

\begin{lemma}\label[lemma]{lem:frame}
For a septic $f(T)=\sum_{i}a_{i}T^{7-i}$ put $q=f'-Tf''+\tfrac13T^{2}f'''-\tfrac1{24}T^{3}f''''$.
Then
\[
  Tq=-\tfrac1{24}\Lop(\Lop-10)f\qquad\text{and}\qquad q=a_{1}T^{5}-a_{3}T^{3}-a_{4}T^{2}+a_{6}.
\]
\end{lemma}

\begin{proof}
Since both operators are linear, it suffices to confirm the equation on monomials. We compute that the eigenvalue of $T^m$ on the left side is  
\[ 
m-m(m-1)+\tfrac13m(m-1)(m-2)-\tfrac1{24}m(m-1)(m-2)(m-3)=-\tfrac1{24}m(m-2)(m-5)(m-7).
\]
Similarly, on the right side, we get the eigenvalue
\[m(7-m)\bigl(m(7-m)-10\bigr)=m(m-2)(m-5)(m-7).\]
The two sides therefore agree, the factor $-\tfrac1{24}$ in
the statement accounting for the difference. The second equation follows after computing the above eigenvalues for $m=7,6,\ldots, 0$. 
\end{proof}

The coefficients $a_{1},a_{3},a_{4},a_{6}$ that survive in $q$ we call \emph{visible},
the others \emph{invisible}; the operator $\Lop(\Lop-10)$ kills exactly the eigenspaces with eigenvalues
$0$ and $10$, that is, the monomials $T^0, T^2, T^5$ and $T^7$. The hypothesis $D_{2}=\lambda XZ$ says nothing directly about $T_{1},C_{1},P_{5}$, which
are built from $D_{1}$ and $D_{3}$. The following three identities express them through
$D_{2}$.

\begin{lemma}\label[lemma]{prop:recoup}
For every binary septic, identically in $a_{0},\dots,a_{7}$,
\[
  T_{1}=-20\,(f,D_{2})_{2},\qquad C_{1}=\tfrac{21}{2}\,(D_{2},D_{2})_{2},\qquad
  P_{5}=-\tfrac{25}{2}\,(f,D_{2}^{2})_{4}.
\]
Consequently, with $H=D_{2}$\footnote{Here $H$ denotes the quadratic covariant $D_2$, of
order two. It is not the Hessian, which for a binary septic is
$(f,f)_2=D_3$, of order ten.}  and $H_{0}=H/207360=AX^{2}+BXZ+CZ^{2}$,
\[
  Q_{f}=\tfrac{175}{2}\bigl(24\,(H,H)_{2}\,(f,H)_{2}-H\,(f,H^{2})_{4}\bigr),\qquad
  \widehat Q_{f}=\tfrac{1}{2880}\bigl(24\,(H_{0},H_{0})_{2}\,(f,H_{0})_{2}-H_{0}\,(f,H_{0}^{2})_{4}\bigr).
\]
\end{lemma}

\begin{proof}
Each identity compares two covariants of the same bidegree: $(3,5)$, $(4,0)$ and $(5,3)$
respectively. By \eqref{eq:dims} the first two spaces are one-dimensional, so each pair is
proportional, and the constant is read off from one septic on which the right-hand side
is nonzero. Take $f=X^{7}+XZ^{6}$, so that 
\begin{align*} 
D_{1} &=2^{7}3^{3}5^{2}7\,X^{4}Z^{2},
& D_{2}&=2^{9}3^{4}5^{2}7\,X^{2},\\
T_{1} &=-2^{13}3^{5}5^{4}7\,XZ^{4},
& (f,D_{2})_{2}&=2^{11}3^{5}5^{3}7\,XZ^{4}.
\end{align*} 
This implies the first equation. (To see the last value: $(X^{7},X^{2})_{2}=0$ since every term needs a $Z$-derivative
of $X^{7}$, and in $(XZ^{6},X^{2})_{2}$ only $j=2$ survives, giving
$\partial_{Z}^{2}(XZ^{6})\,\partial_{X}^{2}(X^{2})=30XZ^{4}\cdot2=60XZ^{4}$; so
$(f,D_{2})_{2}=2^{9}3^{4}5^{2}7\cdot60\,XZ^{4}$.) For the second equation, we let $f=X^{7}+Z^{7}$. In this case 
\[ 
\begin{gathered} 
\begin{aligned} 
D_{2}& =2\cdot5040^{2}\,XZ=:\lambda XZ, & 
D_{3}&=3528\,X^{5}Z^{5},\\ 
T_{5}&=2^{9}3^{5}5^{2}7^{3}(X^{7}-Z^{7}), &
C_{1}&=-2^{9}3^{5}5^{2}7^{3}\lambda, \end{aligned} \\(D_{2},D_{2})_{2}=-2\lambda^{2},\end{gathered} \]
implying the factor of $\tfrac{21}{2}$. For the third identity, we remark that the space $\Cov_{5,3}(7)$ is two-dimensional. Put
$U=(f,D_{2}^{2})_{4}$ and $V=(D_{1},T_{1})_{4}$. On
$f_{1}=X^{7}+XZ^{6}$ and $f_{2}=X^{5}Z^{2}+XZ^{6}$ all three of $P_{5},U,V$ are multiples of
$XZ^{2}$, with factors given in the table below
\[
\begin{array}{c|ccc}
 & P_{5} & U & V\\ \hline
 f_{1} & -2^{23}3^{11}5^{7}7^{2} & 2^{24}3^{11}5^{5}7^{2} & 2^{26}3^{10}5^{6}7^{3}\\
 f_{2} & -2^{25}3^{10}5^{7} & 2^{26}3^{10}5^{5} & 2^{28}3^{8}5^{6}
\end{array}
\]
Since $V/U$ equals $\tfrac{140}{3}$ on $f_{1}$ and $\tfrac{20}{9}$ on $f_{2}$, the
covariants $U,V$ are linearly independent, hence a basis, and $P_{5}=\alpha U+\beta V$.
Evaluating on $f_{1}$ and $f_{2}$ gives $\alpha+\tfrac{140}{3}\beta=-\tfrac{25}{2}$ and
$\alpha+\tfrac{20}{9}\beta=-\tfrac{25}{2}$, so $\beta=0$ and $\alpha=-\tfrac{25}{2}$. Each
entry of the table is confirmed by the script of \Cref{sec:code}, which
verifies the three identities symbolically. For the second part of the statement substitute the three
identities into $Q_{f}=7D_{2}P_{5}-10C_{1}T_{1}$. Remark that  $7\cdot(-\tfrac{25}{2})=-\tfrac{175}{2}$ and
$-10\cdot\tfrac{21}{2}\cdot(-20)=2100=\tfrac{175}{2}\cdot24$. The substitution $H=207360H_{0}$ then contributes a factor
$207360^{3}=2^{27}3^{12}5^{3}$, and $\tfrac{175}{2}\cdot2^{27}3^{12}5^{3}/\cont=1/2880$. As a result $Q_{f}$
is built from $f$ and its quadratic covariant alone. $D_{1},D_{3},T_{5},C_{7}$ are scaffolding.

\end{proof}

\begin{lemma}\label[lemma]{lem:XZ}
For a binary septic $f$, with $f(T)=f(T,1)$, we have that
\[ \begin{gathered}  
(f,XZ)_{2}=-2\,T^{-1}\Lop f \qquad 
(XZ,XZ)_{2}=-2\\
(f,X^{2}Z^{2})_{4}=24\,T^{-2}\Lop(\Lop-6)f.
\end{gathered} \]
\end{lemma}

\begin{proof}
Take $f=X^{m}Z^{7-m}$. In $(f,XZ)_{2}$ only $j=1$ survives, giving
$-2\,m(7-m)X^{m-1}Z^{6-m}$, and $m(7-m)$ is the eigenvalue of $\Lop$ on $T^{m}$. The
second identity is a direct evaluation. In $(f,X^{2}Z^{2})_{4}$ only $j=2$ survives, with
coefficient $\binom42\cdot4\cdot m(m-1)(7-m)(6-m)$, and $m(m-1)(7-m)(6-m)=\ell(\ell-6)$ for
$\ell=m(7-m)$.
\end{proof}

Using the above two lemmas, we are now able to collapse the expression for $Q_f(T,1)$ using $q$ defined in \Cref{lem:frame}.

\begin{proposition}\label[proposition]{prop:collapse}
If $f$ is in normal position with $D_{2}=\lambda XZ$, then \[Q_{f}(T,1)=50400\,\lambda^{3}\,q(T)\]
with $q$ as in \Cref{lem:frame}.
\end{proposition}

\begin{proof}
By \Cref{prop:recoup} and \Cref{lem:XZ}, we compute directly that
\begin{align*}
T_{1}&=-20\lambda(f,XZ)_{2}=40\lambda\,T^{-1}\Lop f,\\ C_{1}&=\tfrac{21}{2}\lambda^{2}(-2)=-21\lambda^{2},\\
P_{5}&=-\tfrac{25}{2}\lambda^{2}(f,X^{2}Z^{2})_{4}=-300\lambda^{2}T^{-2}\Lop(\Lop-6)f.
\end{align*}Hence
\begin{align*} 
  Q_f(T,1) &= 7D_{2}P_{5}-10C_{1}T_{1}=-2100\lambda^{3}T^{-1}\Lop(\Lop-6)f+8400\lambda^{3}T^{-1}\Lop f\\
  &=-2100\lambda^{3}\,T^{-1}\Lop(\Lop-10)f=50400\lambda^{3}q
\end{align*} 
by \Cref{lem:frame}.
\end{proof}

Note: the ratio $7:-10$ is exactly what turns $\mathcal{L}-6$ into
$\mathcal{L}-10$.

 Using the four-term $q$ of \Cref{lem:frame}, define
\begin{equation}\label{eq:g}
\begin{aligned}
  g&=Tq^{2}-\tfrac{11}{12}T^{2}qq'-\tfrac1{12}T^{3}qq''+\tfrac14T^{3}(q')^{2}\\
   &=a_{1}^{2}T^{11}+\tfrac54a_{1}a_{4}T^{8}-\tfrac14(17a_{1}a_{6}+a_{3}a_{4})T^{6}+\tfrac54a_{3}a_{6}T^{4}+a_{6}^{2}T,\\
  P_{0}&=\tfrac12T^{6}(\Dop-2)q=\tfrac32a_{1}T^{11}-\tfrac12a_{3}T^{9}-a_{6}T^{6},\\
  R_{0}&=\tfrac12T^{4}(3-\Dop)q=-a_{1}T^{9}-\tfrac12a_{4}T^{6}+\tfrac32a_{6}T^{4}.
\end{aligned}
\end{equation}

\begin{theorem}\label[theorem]{thm:cert}
For every binary septic $f$ with $A$ and $C$ as in \eqref{eq:D2}, we have
\[
  W(f',g)=q^{3}-35a_{1}fR_{0}+AP_{0}+CR_{0},
\]
as an identity in $\mathbb{Q}[a_0, \ldots, a_7, T]$, hence in $F[a_0, \ldots , a_7, T]$  for every field $F$ of
characteristic zero.
Moreover $\deg g\le 11$, with equality if and only if
$a_1\ne 0$; in particular $\deg g\le 12$, which is the required bound by  \Cref{lem:residues}(b).
\end{theorem}

\begin{proof}
Let $\mathcal E$ be the difference of the two sides. We remark that the polynomials $q,g,P_{0},R_{0}$
involve only the visible coefficients $a_1, a_3, a_4,$ and $a_6$. $W(f',g)$ and $fR_{0}$ are linear in $f$ and each
monomial of $A$ and $C$ contains at most one invisible coefficient. Hence $\mathcal E$ is
affine-linear in $(a_{0},a_{2},a_{5},a_{7})$, and it vanishes if and only if its constant
term and its four partial derivatives vanish. Differentiating $\mathcal E$ with respect to
$a_{0},a_{2},a_{5},a_{7}$ gives, after dividing the first identity by $7$ and the second by $5$,
\begin{align*}
 T^{5}(\Dop-6)g&=-5a_{1}T^{7}R_{0}+5a_{6}P_{0}, &
 T^{3}(\Dop-4)g&=-7a_{1}T^{5}R_{0}+a_{4}P_{0}-2a_{6}R_{0},\\
 2(\Dop-1)g&=-35a_{1}T^{2}R_{0}-10a_{1}P_{0}+5a_{3}R_{0}, & 0&=-35a_{1}R_{0}+35a_{1}R_{0},
\end{align*}
each of which is verified by substituting \eqref{eq:g} and comparing coefficients.
For the constant term put $a_{0}=a_{2}=a_{5}=a_{7}=0$, so
$f=f_{0}=a_{1}T^{6}+a_{3}T^{4}+a_{4}T^{3}+a_{6}T$, $A=-2a_{3}^{2}$, $C=-2a_{4}^{2}$, and write
$q=\sum_{k\in S}c_{k}T^{k}$ with $S=\{5,3,2,0\}$, $(c_{5},c_{3},c_{2},c_{0})=(a_{1},-a_{3},-a_{4},a_{6})$. 
With $\Lam=\Dop^{2}-5\Dop+3$, we compute $f_{0}=\tfrac13T\Lam q$, and the constant term becomes
\[
  W(f_{0}',g)=q^{3}-c_{3}^{2}T^{6}(\Dop-2)q-c_{2}^{2}T^{4}(3-\Dop)q-\tfrac{35}{6}c_{5}T^{5}(\Lam q)(3-\Dop)q .
\]
Both sides are cubic forms in $(c_{5},c_{3},c_{2},c_{0})$ and the coefficient of
$c_{k}c_{l}c_{m}$ carries $T^{k+l+m}$. Thus the equation reduces to twenty scalar identities, one for each
multiset $\{k,l,m\}\subseteq S$. The table below lists the left side, the multinomial coefficient
from $q^{3}$, and the total correction from the three other terms.
\begin{center}
\small
\renewcommand{\arraystretch}{1.08}
\begin{tabular}{@{}lrrr@{\qquad}|@{\qquad}lrrr@{}}
\toprule
$\{k,l,m\}$ & LHS & $q^3$ & corr.\ & $\{k,l,m\}$ & LHS & $q^3$ & corr.\\
\midrule
$555$ & $36$    & $1$ & $35$     & $333$ & $0$  & $1$ & $-1$\\
$553$ & $-32$   & $3$ & $-35$    & $332$ & $3$  & $3$ & $0$\\
$552$ & $-99/2$ & $3$ & $-105/2$ & $330$ & $5$  & $3$ & $2$\\
$550$ & $-29/2$ & $3$ & $-35/2$  & $322$ & $3$  & $3$ & $0$\\
$533$ & $0$     & $3$ & $-3$     & $320$ & $6$  & $6$ & $0$\\
$532$ & $47/2$  & $6$ & $35/2$   & $300$ & $3$  & $3$ & $0$\\
$530$ & $117/2$ & $6$ & $105/2$  & $222$ & $0$  & $1$ & $-1$\\
$522$ & $45/2$  & $3$ & $39/2$   & $220$ & $0$  & $3$ & $-3$\\
$520$ & $41$    & $6$ & $35$     & $200$ & $3$  & $3$ & $0$\\
$500$ & $-99/2$ & $3$ & $-105/2$ & $000$ & $1$  & $1$ & $0$\\
\bottomrule
\end{tabular}
\end{center}
We now explain how to compute the entries in detail. Write $f_0=\sum_{k\in S}\sigma_k c_kT^{k+1}$,
where $\sigma_k=\frac13(k^2-5k+3)$ takes the values $(+,-,-,+)$ on
$S=\{5,3,2,0\}$, and 
\[g =\sum_{\ell\le m}g_{\ell m}c_\ell c_mT^{\ell+m+1}.\]
A direct computation shows that 
\[ g_{\ell m} = \begin{cases} \frac{24-10(\ell+m)- (\ell -m)^2 + 4 \ell m}{12} & \ell \neq m\\
\frac{\ell^2 - 5\ell + 6 }{6} & \ell = m\end{cases}.\]
In the product $W(f_0',g)=f_0'g'-f_0''g$ a monomial $c_kc_\ell c_m$ arises
once for each way of splitting the multiset $\{k,\ell,m\}$ as
$\{k\}\uplus\{\ell,m\}$, the index $k$ being the one taken from $f_0'$.
There is one such splitting for each distinct value occurring in the
multiset. Hence the left side (LHS) of row $\{k,\ell,m\}$ is computed by
\[
\sum_{k}\sigma_k\,(k+1)\,(\ell+m+1-k)\,g_{\ell m},
\]
the sum being over the distinct values $k$, with $\{\ell,m\}$ the
complementary pair in each case. \\
For instance, in row $\{5,5,5\}$ there is
one splitting, $\sigma_5=+1$, $(k+1)(\ell+m+1-k)=6\cdot 6=36$ and
$g_{55}=1$, so second entry (LHS) is
$36$. For the right hand side of the equation $q^3$ contributes $1$ and
$-\frac{35}{6}c_5T^5(\Lambda q)(3-\mathcal D)q$ contributes
$-\frac{35}{6}\cdot 3\cdot(-2)=35$, because $\Lambda T^5=3T^5$ and
$(3-\mathcal D)T^5=-2T^5$. \\
Multiple splittings are possible. For instance, in row $\{3,3,0\}$ there are two splittings:
(1) $k=3$ with $\{\ell,m\}=\{3,0\}$ gives
$(-1)\cdot 4\cdot 1\cdot g_{30}=(-1)\cdot4\cdot1\cdot(-\frac54)=5$, and
(2) $k=0$ with $\{\ell,m\}=\{3,3\}$ gives $(+1)\cdot 1\cdot 7\cdot g_{33}=0$,
since $g_{33}=0$. As a result, the left side is $5=3+2$.\\ 
The whole identity is also verified symbolically by
the script in \Cref{sec:code}. The degree statement is read off from \eqref{eq:g}.
\end{proof}

The polynomial $g$ is not guessed. In the full paper, we develop this Rankin--Cohen-type calculus and show that it confirms Hermite's and Joubert's results in degrees $5$ and $6$ by the same recipe. There, we will also explain the occurrence of the numbers $7$ and $10$ in \Cref{lem:frame}.

\begin{proof}[Proof of \Cref{thm:identity}]
The first sum vanishes by \Cref{lem:euler}, since $\deg q_{f}\le5$. For the second, the
vanishing of $\sum_{i}y_{i}^{3}$ is unaffected if $\widehat Q_{f}$ is replaced by
$Q_{f}=\cont\,\widehat Q_{f}$ as in this case all $y_i$ are multiplied by $\cont$. 
Put
$q_f^{\mathrm{un}}(T)=Q_f(T,1)=\cont\,q_f(T)$. Let $\mathcal U$ be the set of septics over $\Fbar$ with
$\Disc(f)\ne0$, $B^{2}-4AC\ne0$ and $\Res(f,D_{2})\ne 0$ (that is, no root of $D_2$ is a root of
$f$, so that $\infty$ will not be a root after the change of
variable).  By construction, this set is Zariski-open. It is also nonempty as $X^{7}+Z^{7} \in \mathcal{U}$, hence dense. For $f\in\mathcal U$, there exists a projective
transformation $\gamma$ with $D_{2}(\gamma\cdot f)=\lambda XZ$ by \Cref{lem:normal}. By
\Cref{lem:inv}(a) it suffices to prove $\sum y_{i}^{3}=0$ for $\gamma\cdot f$, which again
lies in $\mathcal U$, has $\infty$ not among its roots, and may be assumed to be monic. By
\Cref{prop:collapse}, $q_f^{\mathrm{un}}=50400\lambda^{3}q$, so it suffices to show
$\sum_{i}q(r_{i})^{3}/f'(r_{i})^{3}=0$. In normal position $A=C=0$, and \Cref{thm:cert}
gives $W(f',g)=q^{3}-35a_{1}fR_{0}$. Evaluating at a root $r_{i}$ this gives $W(f',g)(r_{i})=q(r_{i})^{3}$.
Therefore
\[
  \sum_{i}\frac{q(r_{i})^{3}}{f'(r_{i})^{3}}=\sum_{i}\frac{W(f',g)(r_{i})}{f'(r_{i})^{3}}=0
\]
by \Cref{lem:residues}(b), since $\deg g\le11\le 2 \cdot 7 -2$.  Finally, by \Cref{lem:inv}(b), $e_3(y)\Disc(f)$ is a polynomial in the
coefficients of monic $f$. Since $e_1(y)=0$, Newton's identity gives
$\sum y_i^3 = 3e_3(y)$, so this polynomial vanishes on the dense open
set of monic $f \in \mathcal{U}$, hence identically. For any separable
monic $f$, $\Disc(f) \neq 0$ then forces $e_3(y) = 0$.
\end{proof}

\section{Proof of the main theorem}\label{sec:mainproof}

Suppose that $E$ is a field extension of $F$ with $[E:F]=7$. Since $7$ is prime, every $\theta\in E\setminus F$
generates $E$, and its minimal polynomial $f_{\theta}$ is separable of degree $7$. 
For such
$\theta$ put $a=q_{f_{\theta}}(\theta)/f_{\theta}'(\theta)$ as in \eqref{eq:formula}. By
\eqref{eq:trsum}, $\tr(a) = \sum_i y_i$ and $\tr(a^3) = \sum_i y_i^3$ where $y_i =q_{f_{\theta}}(r_{i})/f_{\theta}'(r_{i})$,
so \Cref{thm:identity} gives $\tr(a)=\tr(a^{3})=0$, and $c_{1}(a)=c_{3}(a)=0$ by Newton. Since nonzero elements of $F$ have nonzero trace, it is enough to find $\theta$ with $a \ne 0$, that is, \ with $q_{f_{\theta}}(\theta)\ne0$. Since
$\deg q_{f_{\theta}}\le5<7$, this is equivalent to proving $q_{f_{\theta}}\ne0$.

Fix an $F$-basis of $E$, so $E\cong\AA^{7}(F)$. For \emph{every} $\theta\in E$ let
$f_{\theta}(T)=\det(T-m_{\theta}\mid_E)$ be the characteristic polynomial of multiplication by
$\theta$. We remark that for a generator it is the minimal polynomial, and for $\theta\in F$ it is $(T-\theta)^{7}$, for which $q_{f_{\theta}} = 0$, so that in fact $F \cdot 1 \subseteq \mathcal{Z}$;  we nevertheless remove it explicitly so that the word ``generators'' below needs no comment.
Its coefficients are polynomial functions of the coordinates of $\theta$, hence so are the
coefficients of $q_{f_{\theta}}$.
Let $\mathcal{Z}\subseteq\AA^{7}$ be the closed set where they all
vanish, and $\mathcal{W}=\AA^{7}\setminus(\mathcal{Z}\cup F\cdot1)$. To see $\mathcal{Z} \cup F\cdot1 \ne\AA^{7}$ we may pass to $\Fbar$,
where $E\otimes_{F}\Fbar\cong\Fbar^{7}$ and every monic septic arises as $f_\theta$ for some $\theta$. Take
$f=T^{7}+T$, that is\ $X^{7}+XZ^{6}$, separable (its roots are $0$ and the sixth roots of $-1$).
The values computed in the proof of \Cref{prop:recoup} give $D_{2}=2^{9}3^{4}5^{2}7\,X^{2}$ and
$P_{5}=-2^{23}3^{11}5^{7}7^{2}\,XZ^{2}$, while $C_{1}=\tfrac{21}{2}(D_{2},D_{2})_{2}=0$ because $(X^{2},X^{2})_{2}=0$. Thus this septic lies outside the open set
$\mathcal{U}$ used in the proof of \Cref{thm:identity}, since
$D_2$ is a perfect square. That theorem has already been proved for
every separable septic; here we only need $Q_f \neq 0$. We deduce that

 $Q_{f}=7D_{2}P_{5}=-2^{32}3^{15}5^{9}7^{4}\,X^{3}Z^{2}$ and
\[
  q_{f}=Q_{f}(T,1)/\cont=-3\cdot35^{3}\,T^{3}\ne0 .
\]
So $\mathcal{Z}$ is a proper closed subset and $\mathcal{W}$ is a nonempty open set consisting of generators.
Since $F$ is infinite, $F$-points are dense in $\AA^{7}$, so $\mathcal{W}(F)$ is nonempty (indeed
Zariski-dense), and every $\theta\in \mathcal{W}(F)$ gives $a\ne0$, hence a generator with the desired vanishing traces $\tr(a) = \tr(a^3) =0$. \qed

\begin{remark}
Three simplifications are special to the proof given in this announcement. First, characteristic zero allows the rational constants in the Wronskian normal form and makes Newton's equivalence between the third power sum and $e_3$ available. Second, every characteristic-zero field is infinite, so the required Zariski-open subset has an $F$-point. Third, because the extension has prime degree, any nonzero trace-zero output is automatically primitive. The comprehensive paper replaces the rational computations by calculations of a primitive integral covariant, proves the coefficient identity in every characteristic, treats arbitrary degree-seven \'etale algebras over infinite fields, and degree-seven field extensions over arbitrary fields.
\end{remark}

\section[Complete Numerical Examples]
{Complete Numerical Examples\texorpdfstring{\protect\footnotemark}{}}
\label{sec:example}

\footnotetext{The relevant SageMath code for the computations in this numerical example can be found here:
\mbox{\url{https://sagecell.sagemath.org/?q=brkqtq}}. 

The user can use this calculator to experiment with other polynomials. However, note that while the theorem guarantees the covariant evaluates to a nonzero generator on a dense Zariski-open set, an unlucky choice of the initial generator may give zero output. In that case, choose another generator of the same field extension and apply the construction to its minimal polynomial; an example is given below. }

Consider the polynomial
\[
  f_0(T)=T^7+2T^6+2.
\]
Using Eisenstein's criterion with $p=2$, it is easily seen that the polynomial is irreducible over $\mathbb{Q}$. Translating the variable by one gives a new polynomial 
\begin{equation}\label{eq:worked-g}
\begin{aligned}
  f_1(T)=f_0(T-1)
  ={}&T^7-5T^6+9T^5-5T^4-5T^3 + 9T^2-5T+3.
\end{aligned}
\end{equation}
Since translation preserves irreducibility,  $f_1$ is also irreducible over $\QQ$. Let $\theta$ be a root of $f_1$ in the algebraic closure and put $E=\QQ(\theta)$. Notice that the two coefficients we want to remove are both present in $f_1$: the coefficients of $T^6$ and $T^4$ are both $-5$.

The homogenization of $f_1$ is
\[
\begin{aligned}
  G(X,Z)={}&X^7-5X^6Z+9X^5Z^2-5X^4Z^3-5X^3Z^4 +9X^2Z^5-5XZ^6+3Z^7.
\end{aligned}
\]
For this input, the transvectant chain gives
\[
  Q_G=-12330752933527289856000000000\, Z^2(X-Z)^2(7X-4Z).
\]
Dividing by the content  $\cont=2^{32}3^{14}5^6\cdot 7$ gives
\[ q_{f_1}(T) = -5488000\,(T-1)^2(7T-4),
\qquad 5488000=2^7 5^3 7^3 . \]
(Equivalently, $q_{f_0}(T)=-5488000\,T^2(7T+3)$ and
$q_{f_1}(T)=q_{f_0}(T-1)$: translation acts by a unipotent
matrix, so the covariant is carried along with no scalar
factor.)

Since a nonzero rational scalar changes neither $\tr(a)$ nor
$\tr(a^3)$ nor the property of generating $E$, we may work
with
\[\tilde{q}_{f_1}(T) = -(T-1)^2(7T-4) \]
in place of $q_{f_1}$.

It can be checked that 
\[
  f_1'(T)=7T^6-30T^5+45T^4-20T^3-15T^2+18T-5.
\]
Now we define
\[
  a := \frac{\tilde{q}_{f_1}(\theta)}{f_1'(\theta)} = \frac{-(\theta-1)^2(7\theta-4)}
  {7\theta^6-30\theta^5+45\theta^4-20\theta^3-15\theta^2+18\theta-5}.
\]
The denominator is nonzero because $f_1$ has distinct roots. The numerator is also nonzero because a degree-seven algebraic number cannot equal  $1$ or $4/7$.
An exact resultant calculation gives
\[
\begin{aligned}
  \chi_a(Y)={}&Y^7
  -\frac{1350885}{7619054}Y^5
  -\frac{282963}{7619054}Y^3 +\frac{178605}{3809527}Y^2
  +\frac{89867}{30476216}Y
  -\frac{1655105}{60952432}.
\end{aligned}
\]
Note that the coefficients of both $Y^6$ and $Y^4$ have been eliminated, as desired. 

Moreover, the element $a$ is nonzero and has trace zero, so it is not rational. Since $[E:\QQ]=7$ is prime, $a$ generates $E$.

We remark that it is possible that the computed element $a$
is zero. For example, take
\[ f(T)=T^7-6T^6+3T^5+5T^4+10T^3-4T^2+4T-2, \]
which is irreducible over $\QQ$, and let $\theta$ be a root
of $f$. One checks from \eqref{eq:D2} that $A=B=C=0$, so $D_2=0$ and
hence $T_1=C_1=P_5=0$ and $Q_f=0$ by \Cref{prop:recoup}. We replace
$\theta$ by $\theta'=\theta+\theta^2$; since $\theta'\notin\QQ$
and $7$ is prime, $\QQ(\theta')=\QQ(\theta)$. This element has
minimal polynomial
\[ f_2(T)=T^7-36T^6+125T^5+245T^4+180T^3-10T^2-60T-50, \]
which is irreducible over $\QQ$, with
$\Disc f_2=2^4 5^7 53^4\cdot 3230449^2$. Its $T^6$ and $T^4$
coefficients, $-36$ and $245$, are again both nonzero, and
\begin{align*}
q_{f_2}(T) &= 2^3 5^5 53^3\,\widetilde q_{f_2}(T)\neq 0,\\
\widetilde q_{f_2}(T) &= 3440233\,T^5-3259080\,T^4+33296830\,T^3\\
&\qquad +12058835\,T^2-50388855\,T-10419925 .
\end{align*}
We compute that
\[ a'=\frac{\widetilde q_{f_2}(\theta')}{f_2'(\theta')}\neq 0 \]
with $\tr(a')=\tr(a'^3)=0$; since $a'\neq 0$ has trace zero it
is not rational, so it generates $E$.

\section{Supplementary Code}\label{sec:code}

The accompanying code allows the reader to reproduce the longer polynomial calculations using exact integer and rational arithmetic. The repo is available at \url{https://github.com/kindaai/HJ7}.

The main script \texttt{verify\_hj7.py} constructs

\[
  D_{1},D_{2},D_{3},T_{1},T_{5},C_{1},C_{7},P_{5},
  \qquad
  Q_{f}=7D_{2}P_{5}-10C_{1}T_{1},
\]

from the unnormalised transvectant of \Cref{def:tv}. It reproduces the following calculations.

\begin{enumerate}

\item
It expands $Q_{f}$, obtaining order $5$, the $780$ nonzero monomials, and
\[
  2^{32}3^{14}5^{6}\cdot7
\]
as the greatest common divisor of its coefficients. It also reproduces the formula \eqref{eq:D2} for $D_{2}$ and the formula for $Q_{f}$ in
\Cref{prop:recoup}.

\item
It reproduces the three recoupling identities of \Cref{prop:recoup}, the four dimensions in \eqref{eq:dims}, the three formulas of \Cref{lem:XZ}, and the two formulas of \Cref{lem:frame}. It also reproduces the values displayed in the proof of \Cref{prop:recoup} and the calculation in \Cref{prop:collapse} for a septic in normal position.

\item
It expands the auxiliary polynomial $g$ in \eqref{eq:g} and computes both sides of the identity in \Cref{thm:cert} as polynomials in $a_{0},\dots,a_{7}$ and $T$. It also repeats the four derivative calculations in the proof and the calculations in all twenty rows of the table.

\item
It reproduces the calculation with $T^{7}+T$ in \Cref{sec:mainproof}. It also reproduces the displayed translation, covariant, derivative, and resultant in the first numerical example of \Cref{sec:example}, beginning with \eqref{eq:worked-g}.

\end{enumerate}

A second script, \texttt{verify\_from\_q.py}, reproduces the first numerical example without constructing $Q_{f}$ and without using transvectants. It begins with the translated septic $f_{1}(T)$ of \eqref{eq:worked-g} and the polynomial
\[
  \tilde{q}_{f_{1}}(T)=-(T-1)^{2}(7T-4)
\]
displayed in \Cref{sec:example}. Letting $\theta$ be the class of $T$ in
$\QQ[T]/(f_{1})$, the script forms
\[
  a=\frac{\tilde{q}_{f_{1}}(\theta)}{f_{1}'(\theta)}.
\]
It computes the inverse of $f_{1}'(\theta)$ and then computes
\[
  \tr(a),
  \qquad
  \tr(a^{3}).
\]
It also computes the characteristic polynomial of $a$ and compares it with the polynomial displayed in \Cref{sec:example}. Finally, it shows that the $Y^{6}$ and $Y^{4}$ coefficients are zero, computes irreducibility over $\QQ$, and compares the characteristic polynomial with the polynomial obtained from the resultant. Thus the numerical example is reproduced once from the transvectant chain and once directly from $q_{f_{1}}$.

\section{Declaration on the use of artificial intelligence}\label{sec:ai}
This project originated with Behzad Nikzad, who initiated an AI-assisted exploration of permutations of Newton sums and their relations. He subsequently brought together the present research team, and the present work developed through our collaboration.
 During this work the authors used multiple models from ChatGPT (OpenAI) and Claude (Anthropic) for extensive collaboration on technical details, computer-assisted verification, and exposition. The models also took part in the search that led to the covariant $Q_{f}$ and to
the auxiliary polynomial $g$. The authors take full responsibility for the results presented in this paper.

\end{document}